%% file: AESks_rev2.tex
\documentclass[oneside,reqno]{amsart}
\input{files/preamble}

\begin{document}
\title[On the primitivity of the AES-128 key-schedule]{On the primitivity of the AES-128 key-schedule}

 \author[R.~Aragona]{Riccardo Aragona}
 \address[R.~Aragona]{DISIM \\
 Universit\`a degli Studi dell'Aquila\\
 via Vetoio\\
 67100 Coppito (AQ)\\
 Italy}       
\email{riccardo.aragona@univaq.it}

\author[R.~Civino]{Roberto Civino}
 \address[R.~Civino]{DISIM \\
 Universit\`a degli Studi dell'Aquila\\
 via Vetoio\\
 67100 Coppito (AQ)\\
 Italy}      
 \email{roberto.civino@univaq.it} 
 
\author[F.~Dalla Volta]{Francesca Dalla Volta}
 \address[F.~Dalla Volta]{Dipartimento di Matematica e Applicazioni \\
 Universit\`a degli studi di Milano - Bi\-coc\-ca\\
Piazza dell'Ateneo Nuovo, 1 \\
20126 Milano (MI) \\
 Italy}      
\email{francesca.dallavolta@unimib.it}

\date{} \thanks{All the authors are members of INdAM-GNSAGA
 (Italy). This work was partially supported by the Centre of EXcellence on Connected, Geo-Localized and 
 Cybersecure Vehicles (EX-Emerge), funded by Italian Government under CIPE resolution n. 70/2017 (Aug. 7, 2017).}


\begin{abstract}
The key-scheduling algorithm in the AES is the component 
responsible for selecting from the master key the sequence of round keys 
to be xor-ed to the partially encrypted state at each 
iteration. We consider here the group $\Gamma$ generated by the 
action of the AES-128 key-scheduling operation, and we prove 
that the smallest group containing $\Gamma$ and all the translations of the message space is primitive. As a 
consequence, we obtain that no linear partition of the message space can be invariant under its action. 
\end{abstract}

\maketitle


\section{Introduction}
The encryption functions of the AES are the composition of a sequence of round transformations made by a confusion and a diffusion layer followed by a key addition with the so-called round key, derived by
the user master key by means of the public key-scheduling algorithm. 
While the confusion-diffusion step has been designed to provide long-term resistance against known and possibly future attack, the key-schedule has 
been chosen also without neglecting the application of the cipher in resource-constrained devices. The necessary efforts  to keep  the encryption lighter~\cite{aes}
made \emph{de facto} the confusion-diffusion step almost completely in charge of the security. 
Although some recent improvements in the AES cryptanalysis are based on structural properties of the SPN design~(e.g.\ \cite{ronjom2017yoyo,bardeh2019exchange,dunkelman2020retracing}), 
unsurprisingly, also the key-schedule has been targeted in various attacks in recent years~\cite{biryukov2009related,mala2010improved,boura2018making}.
In general, key-scheduling algorithms seem the component on which there is the least consensus on general 
design criteria  and arguably the components for which attacks are less standardised. 

Despite two decades of cryptanalysis, only recently Leurent and 
Pernot have shown the existence of an invariant subspace for four rounds of the AES-128 key-schedule~\cite{leurentnew}.
Such a finding allowed the authors to provide an alternative representation of the key-schedule as four independent
actions on each of the 4-byte-word components of the round key. 
Although related only to the key-schedule, the result is then used  to obtain global improvements in already known differential attacks, showing how the subspace analysis of the key-schedule may highlight some subspace structures that interact with similar structures in the main round function inducing security flaws. 
 
Initially, the more general idea of finding subspaces which are invariant under the encryption functions, for some or possibly all the keys,
 has been notably exploited by Leander et al.\ to cryptanalyse PRINTcipher~\cite{leander2011cryptanalysis}.
 The above-mentioned strategy make use of the fact that an entire subspace
of the message space (or of the key space) is not moved by the encryption functions. Subspace trail cryptanalysis~\cite{grassi2016subspace}, a generalization of invariant
subspace cryptanalysis, has been also used to attack reduced-round AES~\cite{grassi2017new}.

The imprimitivity attack, introduced 
by Paterson against an intentionally flawed but apparently secure DES-like block cipher~\cite{paterson1999imprimitive},  is conceptually similar to invariant subspaces attacks, 
except it exploits the existence of a full \emph{partition} of the message space that is preserved by the encryption.
 In particular, in this attack scenario, the cryptanalyst usually takes advantage of an entire \emph{linear partition} of the message space, i.e.\ a partition made by the cosets of a proper and non-trivial subspace, which is invariant.
While it is hard in general to prove the non-existence of invariant subspaces (see~\cite{beierle2017proving} for an analysis of the security impact provided by the choice of the round constants), the non-existence of invariant linear partitions after one round can be more easily established using group-theoretical arguments, i.e.\ proving that a given group containing the encryption functions acts primitively on the message space. 

Non-existence results for invariant partitions in standardized constructions have been proved in the last years~\cite{wendes,sparwenrij,wenkas,aragona2017group}, and more general results determining conditions which imply the non-existence of invariant linear partitions obtained by primitivity arguments can be found in the 
literature~\cite{carantiprimitive,aragona2018primitivity}.\\

 In this work we prove a \emph{primitivity} result on the AES-128 key-schedule (see Theorem~\ref{thm:main} and Corollary~\ref{main:coro}), i.e.\ we show that no linear partition can be invariant after one round, when each possible
vector is considered as round counter.
The strategy used here is the following: the action of the key-schedule is
modeled by means of a formal operator defining a group which is proved to be primitive using Goursat's Lemma (cf. Theorem~\ref{gours}). In particular, we 
prove that the group generated by the action of the AES-128 key-schedule is primitive provided that a suitable considerably smaller group
generated by some AES-128 components is primitive, a result which can be established using known facts~\cite{carantiprimitive,aragona2018primitivity}.
As a consequence, our result can be generalized to each substitution-permutation network whose round components
are suitable for generating a primitive group~\cite{carantiprimitive} and whose 4-branch AES-like key-schedule is built accordingly.
To our knowledge, with respect to invariant partitions, the study carried out in this paper is the first example of group-theoretical investigation of the sole key-schedule, which is 
in general excluded from primitivity arguments, except for some recent partial results~\cite{arafeistel,calderinikey}.

\subsubsection*{Related works}
In this paper we use the strategy of a \emph{primitivity reduction} via Goursat's Lemma. We show indeed that the primitivity of a complex structure, such as the 4-branch key-schedule transformations of AES, is inherited from the primitivity of the group generated by simpler SPN-like functions, i.e.\ those acting on the last group of bytes. Similar arguments have been used to prove that the primitivity of more complex structures (e.g.\ Feistel networks, Lai-Massey constructions) reduces to the primitivity of their inner SPN-like components~\cite{aragonawave,aragona2020primitivity}.

\subsubsection*{Organization of the paper}
In Section~\ref{sec:pre} we introduce the notation and the preliminary results, and  present an algebraic representation 
 of the AES-128 key-schedule and the related permutation group. In Section~\ref{sec:pri} we present our primitivity reduction in Theorem~\ref{thm:main} and show its application to AES  in Corollary~\ref{main:coro}. The technical proof of Theorem~\ref{thm:main} with the use of Goursat's Lemma is shown in Sec.~\ref{sec:proof}.
 Finally, in Section~\ref{sec:concl} we draw our conclusions.


\section{Preliminaries and model}\label{sec:pre}
In this section we introduce some notation and preliminary results, starting by briefly recalling the definition of the
AES-128 key-schedule. The reader is invited to refer to Daemen and Rijmen for a detailed description including comments on 
design choices~\cite{aes}.\\

The AES-128 key-schedule is an invertible function of $\Sym(\F_2^{128})$ which, using the cipher's components, transforms the previous
round key into the next one, starting from the master key, proceeding as shown in Fig.~\ref{fig:AESks}, where 
\begin{itemize}
\item[-] $\lambda: \F_2^{32} \rightarrow \F_2^{32}$ denotes the linear operation \texttt{RotWord},
\item[-] $\gamma : \F_2^{8} \rightarrow \F_2^{8}$ represents the AES S-Box \texttt{SubBytes},
\item[-] $rc_i \in \F_2^8$ is a round constant different in each round.
\end{itemize}
In particular, round-key bits are gathered into four groups, each made by 4 bytes. The bytes of the last group are first shifted left by one position and then 
transformed by the cipher's S-Box. Finally, a round-dependent counter is xor-ed to the first byte of the last block. The output of this transformation is then xor-ed to
the remaining three blocks of bytes as shown in Fig.~\ref{fig:AESks}. 

\begin{figure}
\centering
\includegraphics[scale= 0.47]{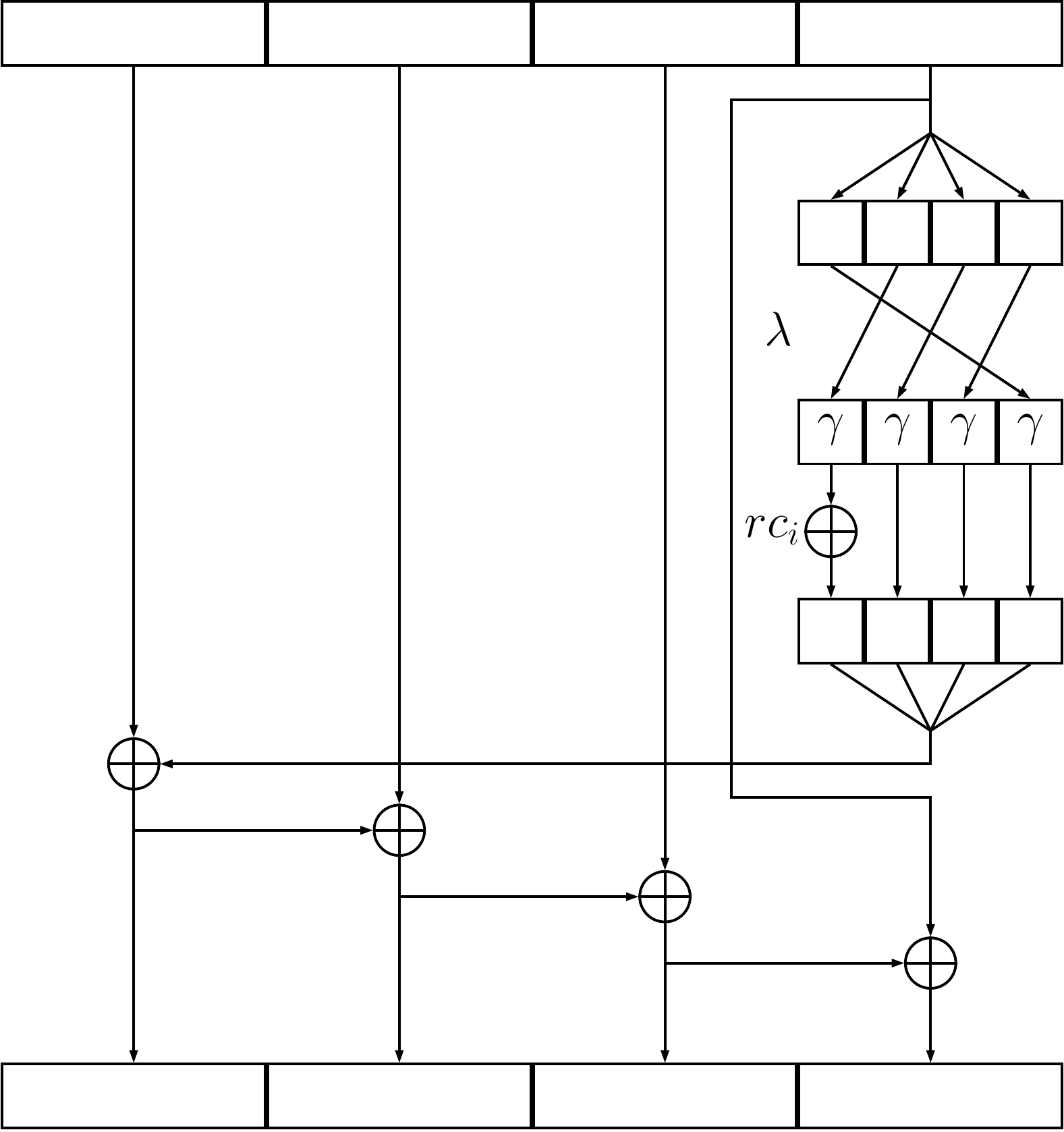}
\caption{The $i$-th transformation of the AES-128 key-schedule}\label{fig:AESks}
\end{figure}

\subsection*{Notation}
In the following, $n$ is a non-negative integer and  $V \deq \F_2^{n}$ is the $n$-dimensional vector space over $\F_2$. If $H$ is a subspace of $V$ we write $H \leq V$, and the same notation is used to denote subgroups. We denote by $\mbb 0 : V \rightarrow V$ the null function on $V$. Moreover, $\Sym(V)$ denotes the symmetric group acting on $V$ and $\mbb1$ its identity. If $f \in \Sym(V)$ and $x \in V$ we write $xf$ to denote the functional evaluation $f(x)$. The group of the translations on V, i.e.\ the group of the maps $\sigma_v: V \rightarrow V$, such that $x\mapsto x+v$, is denoted by $T_{n}$. We also denote by $\AGL(V)$ the group of all affine permutations of $V$ and by $\GL(V)$ the group of the linear ones.
In particular, in the case under investigation,    we mean by $n$ the size of each group of $4$ bytes, i.e.\ $n=32$ bits. In this assumption, the key-scheduling transformation will be acting on $V^4$ as an element of the symmetric group $\Sym(V^4)$, whose corresponding 
 group of translations is denoted by $T_{4n}$, where the translation $\sigma_{(v_1,v_2,v_3,v_4)}$ acts on $(x_1,x_2,x_3,x_4) \in V^4$ as 
\[
(x_1,x_2,x_3,x_4)\mapsto (v_1+x_1,v_2+x_2,v_3+x_3,v_4+x_4).
\]
It is worth noting here that the addition with the round counter in the AES-128 key-schedule acts exactly as a particular translation of $T_{4n}$.

For sake of clarity, we will use different notations for elements of $V$, $V^2$ and $V^4$. In particular, we will denote an element of $V^4$ by superscripting an arrow on the symbol, i.e.\ $\vec v \in V^4$, an element of $V^2$ using symbols in bold, i.e.\ $\vec v=(\bm{v_1},\bm{v_2})$, in such a way \[\vec v=(\bm{v_1},\bm{v_2})=(v_1,v_2,v_3,v_4)\in V^4,\] where $\bm{v_i}\in V^2$ and $v_j\in V$ for $1 \leq i \leq 2$ and $1 \leq j\leq 4$. \\

Let us now introduce the elements of group theory used throughout this article.
\subsection*{Groups}
Let $G$ be a group acting on a set $\M$. For each $g \in G$ and $v \in \M$ we denote the action of $g$ on $v$ as $vg$. 
The group $G$ is said to be \emph{transitive} on $\M$ if for each $v,w \in \M$ there exists $g \in G$ such that $vg=w$.
A partition $\mathcal{B}$ of $\M$ is \emph{trivial} if $\mathcal{B}=\{\M\}$ or $\mathcal{B}=\{\{v\} \mid v \in \M\}$, and \emph{$G$-invariant} if for any $B \in \mathcal{B}$ and $g \in G$ it holds $Bg \in \mathcal{B}$. Any non-trivial and $G$-invariant partition $\mathcal{B}$ of $\M$ is called a \emph{block system} for $G$. In particular any $B \in \mathcal{B}$ is called an \emph{imprimitivity block}. The group $G$ is \emph{primitive} in its action on $\M$ (or $G$ \emph{acts primitively} on $\M$) if {$G$ is transitive and} there exists no block system. 
Otherwise,  the group $G$ is \emph{imprimitive} in its action on $\M$ (or $G$ \emph{acts imprimitively} on $\M$). We recall here some well-known results that will be useful in the remainder of this paper~\cite{cameronpermutation}. 

\begin{lemma}\label{lemma:trans}
  If $T \leq G$ is transitive, then a block system
  for $G$ is also a block system for $T$.
\end{lemma}
In the case under consideration in this paper, the block system will be a \emph{linear partition}:

\begin{lemma}\label{translatioBlocks}
Let $\M$ be a finite vector space over $\mathbb F_2$ and $T$ its translation group. 
Then $T$ is transitive and imprimitive on $\M$. A block system $\mathcal U$ for $T$ is composed by
  the cosets  of a non-trivial and proper subgroup $ U < (M,+)$, i.e.\ 
  \begin{equation*}
 \mathcal U =    \{ 
     U + v
      \mid
 	v \in    M
      \}.
        \end{equation*}
  \end{lemma}
  \subsection*{The key-schedule representation}
Let us now introduce the representation of the AES-128 key-schedule that allows us to provide an easy description of the subgroup of $\Sym(V^4)$ which is
  the subject of this work. Let us start by defining the transformation acting on the last group of four bytes, as in Fig~\ref{fig:AESks}.

\begin{definition}\label{def:opaes}
Let $\rho_{\mathrm{AES}} \in \Sym(V)$ be the composition of $\lambda$ and the parallel application of $4$ copies of $\gamma$, i.e.\ 
\[
\rho_{\mathrm{AES}}  \deq \lambda\gamma' \in \Sym(V),
\]
where $\gamma': \F_2^{32} \mapsto \F_2^{32}$, $(v_1,v_2,v_3,v_4) \mapsto (v_1\gamma,v_2\gamma,v_3\gamma,v_4\gamma)$, with $v_i \in \F_2^8$.
\end{definition}
The function previously defined, up to the xor with the round counter in the first byte, represents the transformation acting on the last group of bytes in the
AES-128 key-schedule. The following definition is instead a more general description of the full transformation.
\begin{definition}\label{def:op}
Given $\rho \in \Sym(V)$, let us define the \emph{AES-like key-schedule operator induced by $\rho$} as the formal matrix
\[
\rb \deq \begin{pmatrix}
\mbb1 & \mbb1 & \mbb1 & \mbb1\\
\mbb0 & \mbb1& \mbb1 & \mbb1\\
\mbb0 & \mbb0& \mbb1 & \mbb1\\
\rho & \rho & \rho & \mbb1+\rho
\end{pmatrix},
\] 
acting on $V^4$ as 
\[(v_1,v_2,v_3,v_4)\mapsto (v_1+v_4\rho,v_1+v_2+v_4\rho,v_1+v_2+v_3+v_4\rho, v_1+v_2+v_3+v_4+v_4\rho),\]
as also displayed in Fig.~\ref{fig.rounds}.
The operator $\rb$ has the following inverse acting as 
\[
(v_1,v_2,v_3,v_4)\rb^{-1}=(v_1 + (v_3+v_4)\rho,v_1+v_2,v_2+v_3,v_3+v_4).
\]
\end{definition}
\begin{figure}
\centering
\includegraphics[scale= 0.55]{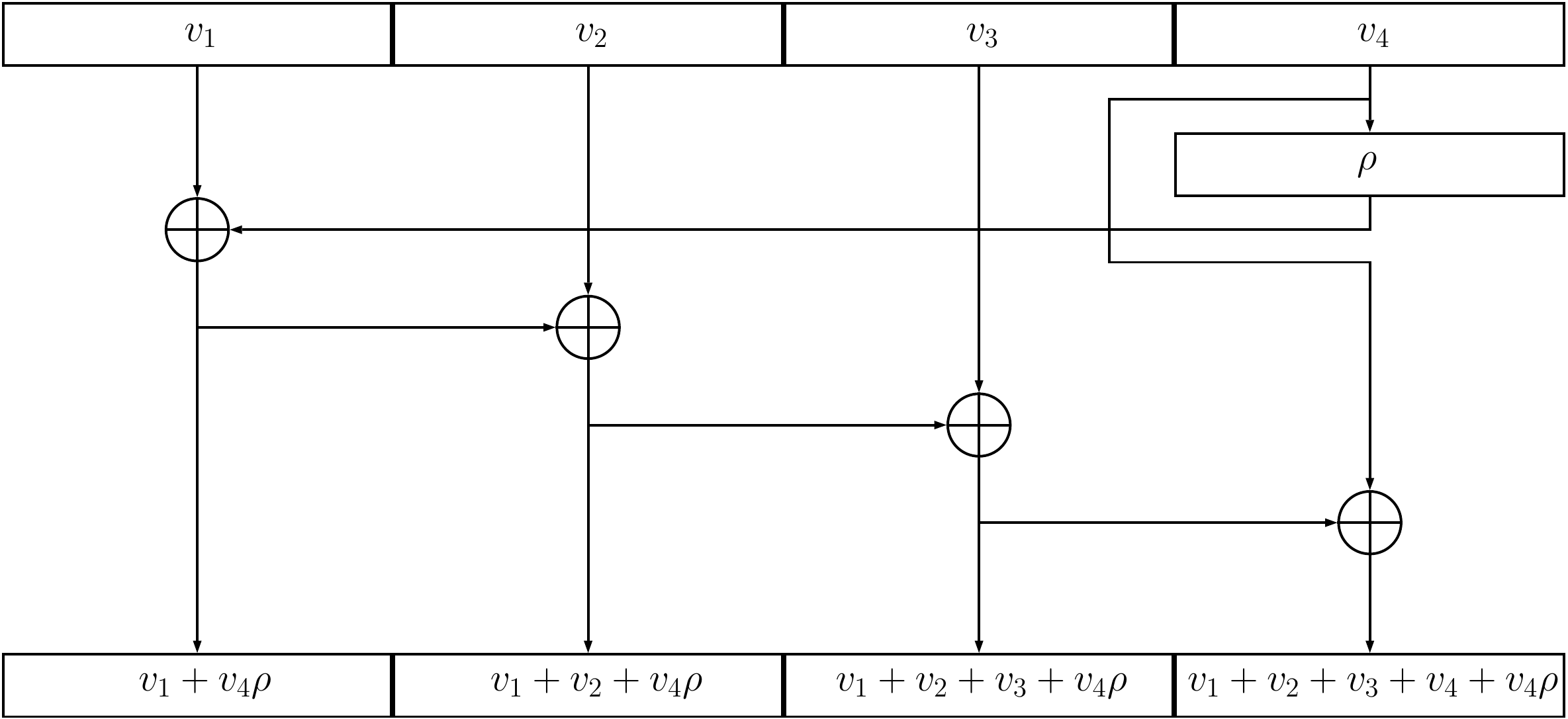}
\caption{The key-schedule operator induced by $\rho$}\label{fig.rounds}
\end{figure}

It is not hard to notice, when considering $\rho_{\mathrm{AES}}$, that the map \[\overline{\rho_{\mathrm{AES}}}\sigma_{(\overline{rc_i},\overline{rc_i},\overline{rc_i},\overline{rc_i})}\]
correspond to the $i$-th round-key transformation in the AES-128 key-schedule, where $\overline{rc_i} = (rc_i, 0,0,0) \in \F_2^{32}$. Keeping in mind that our focus is to study group-theoretical 
properties of the subgroup $\Gamma <\Sym(V^4)$ generated by the elements of the type of $\overline{\rho_{\mathrm{AES}}}\sigma_{(\overline{rc_i},\overline{rc_i},\overline{rc_i},\overline{rc_i})}$, 
for each admissible value of $rc_i \in \F_2^8$, and establish its primitivity  by using Lemma~\ref{translatioBlocks}, it is important to notice that $\Gamma$ does not contain
the whole translation group $T_{128}$. For this reason, $\Gamma$ needs to be extended by assuming a more general action of the round counter. 
\begin{definition}\label{def:gr}
Let us define the group
\[
\Gamma_{\mathrm{AES}} \deq \Span{\overline{\rho_{\mathrm{AES}}}\sigma_{(x,y,z,t)} \mid (x,y,z,t) \in V^4}.
\]
It is easily noticed that 
\begin{itemize}
\item[-] $\Gamma_{\mathrm{AES}}$, which contains $\Gamma$, is the smallest subgroup of the symmetric group containing both $T_{128}$ and the transformation of the AES-128 key-schedule, when the correct round counter is chosen;
\item[-] $\Gamma_{\mathrm{AES}} = \Span{\overline{\rho_{\mathrm{AES}}}, T_{128}}$.
\end{itemize} 

 \end{definition}
In the remainder we prove that $\Gamma_{\mathrm{AES}}$ is primitive. This guarantees that no non-trivial and proper subgroup
$U < V^4$ can generate a partition, as in Lemma~\ref{translatioBlocks}, which is invariant under the transformations of $\Gamma_{\mathrm{AES}}$.


\section{The primitivity of $\Gamma_\mathrm{AES}$}\label{sec:pri}
In this section we prove our main result, i.e.\ the primitivity of $\Gamma_{\mathrm{AES}}$ (cf.\ Corollary~\ref{main:coro}), as a consequence of
a more general result (cf.\ Theorem~\ref{thm:main}) in which we show that the primitivity of $\Span{\overline{\rho}, T_{4n}}$ \emph{reduces} to the primitivity of $\Span{{\rho}, T_n}$,
when $\overline{\rho}$ is the key-schedule operator induced by $\rho$ (cf.\ Definition~\ref{def:op}) and provided that $\rho$ is bijective and not affine.
The proof of the primitivity reduction is rather technical and makes use of repeated applications of Goursat's Lemma (see Sec.~\ref{sec:proof}) so, for the sake of readability, is shown in a separate section.  

We can anticipate our main contribution which is stated as follows: 
\begin{theorem}[Primitivity reduction]\label{thm:main}
Let $\rho \in \Sym(V) \setminus \AGL(V)$.
If $\Span{\rho, T_n}$ is primitive on $V$, then $\Span{\rb,T_{4n}}$ is primitive on $V^4$.
\end{theorem}
\begin{proof}
See Sec.~\ref{sec:proof}.
\end{proof}

According to the previous fact, the primitivity of the AES-128 key-schedule can be deduced by the primitivity of the group generated by $\rho_{\mathrm{AES}}$ and $T_{32}$, 
i.e.\ by the composition of the linear transformation \texttt{RotWord} and the parallel application of 4 copies of the S-Box \texttt{SubBytes} and by the translations on the space of $4$-byte words. As already mentioned, the primitivity of the latter can be obtained, as shown below, from already established results. \\

Let us prove that $\rho = \rho_{\mathrm{AES}}$ satisfies the hypothesis of Theorem~\ref{thm:main}, i.e.\ that $\Span{\rho_{\mathrm{AES}}, T_{32}}$ generates a primitive group. To do so, we need the following definitions
and a general result of primitivity for substitution-permutation networks~\cite{aragona2018primitivity}.

Let us write $n = s \cdot b$, for some $s,b > 1$, and let us decompose $V$ as a direct sum of subspaces accordingly, i.e.\ $V = \bigoplus_{i=1}^b V_i$, where $\dim(V_i) = s$. Each $V_i$, spanned by the canonical vectors $e_{s(i-1)+1}, \ldots, e_{s(i-1)+s}$, is called a \emph{brick}. Recall that, in the case of $\rho_{\mathrm{AES}}$, we have $n=32$, $s = 8$ and $b=4$. 

Given $f: \F_2^s \rightarrow \F_2^s$, for each $a \in \F_2^s$, $a \ne 0$, we denote by
\[\partial_a(f): \F_2^s \rightarrow \F_2^s, \quad x \mapsto xf+(x+a)f\]
the \emph{derivative of $f$ in the direction $a$}.
Recall that when $f \in \Sym(\F_2^s)$ is \emph{$\delta$-differentially uniform}, for some $2 \leq \delta \leq 2^s$, then $\Size{\mathrm{Im}(\partial_a(f))}\geq 2^{s}/\delta$ for each $a \ne 0$.
Moreover, if $0 f = 0$, we say that $f$ is \emph{$\delta$-anti-invariant} if
for any two subspaces $W_1, W_2 \leq  \F_2^s$ such that $W_1 f = W_2$, then either $\dim(W_1) = \dim(W_2) < s - \delta$ or
$W_1 = W_2 = \F_2^s$.

The next theorem is stated using the notation introduced above in this section.
\begin{theorem}[\cite{aragona2018primitivity}]\label{thm:present}
Let $f \in \Sym(\F_2^s)$ such that 0f=0, let $F \in \Sym(V)$ be the function acting as $f$ on each $s$-dimensional brick $V_i$ of $V$ and let $\Lambda \in \GL(V)$. If no non-trivial and proper direct sum of bricks of $V$ is invariant under $\Lambda$ and for some $2 \leq \delta \leq s-1$ the function $f$ is
\begin{itemize}
\item[-] $2^\delta$-differentially uniform,
\item[-] $(\delta-1)$-anti-invariant,
\end{itemize}
then $\Span{F\Lambda, T_n}$
is primitive.
\end{theorem}

We are now ready to prove the primitivity of $\Span{\rho_{\mathrm{AES}}, T_{32}}$ as a consequence of Theorem~\ref{thm:present}.
\begin{theorem}\label{thm:main2}
The group $\Span{\rho_{\mathrm{AES}}, T_{32}} < \Sym(\F_2^{32})$ is primitive.
\end{theorem}

\begin{proof}
Let $\lambda \in \GL(V)$ and $\gamma' \in \Sym(V)$, as in Definition~\ref{def:opaes}, be respectively the \texttt{RotWord} transformation and the parallel application of 4 copies of $\gamma$, the
S-Box \texttt{SubBytes}. It is well know that $\gamma$ is, up to affine transformations, the function which sends $0$ into $0$ and each non-zero element into its multiplicative inverse in $\F_{2^s}$. Such a function is $4$-differentially uniform and $1$-anti invariant, i.e.\ satisfies the hypotheses of Theorem~\ref{thm:present} for $\delta=2$~\cite{nyberg1993differentially,carantiprimitive}. Notice that anti-invariance and differential uniformity are invariant under inversion~\cite{carlet1998codes} and under affine transformations, i.e.\ also $\gamma^{-1}$ satisfies the
hypotheses of Theorem~\ref{thm:present}. Moreover, it easily checked that no non-trivial and proper direct sum of bricks of $V$ is invariant under $\lambda$, and the same trivially holds also for $\lambda^{-1}$. Therefore, from Theorem~\ref{thm:present}, $\Span{(\gamma')^{-1}\lambda^{-1}, T_{32}}$ is primitive, and consequently so is $\Span{\lambda\gamma', T_{32}} = \Span{\rho_{\mathrm{AES}}, T_{{32}}}$.
\end{proof}

The following final conclusion is derived.
 
\begin{corollary}\label{main:coro}
The group $\Span{\overline{\rho_{\mathrm{AES}}},T_{128}}< \Sym(\F_2^{128})$ generated by the transformations of the AES-128 key-schedule is primitive.
\end{corollary}


\section{The primitivity reduction - Proof of Theorem~\ref{thm:main}}\label{sec:proof}
This section is entirely devoted to the proof of Theorem~\ref{thm:main}, which may be skipped entirely from the reader who is not interested in the technical details. 
Despite its apparent complexity, the (repeated) use of Goursat's Lemma, which is introduced below, represents a reasonable way to describe any generic subspace $U$ which is candidate to be a linear block (and which, \emph{a fortiori}, is necessarily  trivial).\\

In order to prove our result, we need to determine a block system for $V^4 = V^2 \times V^2$, i.e.\ the set the cosets 
of a suitable subgroup of $V^2\times V^2$.
This can be accomplished by using the following 
characterization of subgroups of the direct product  of two groups in terms of
suitable sections of the direct factors~\cite{goursat}.  

\begin{theorem}[Goursat's Lemma]\label{gours}
  Let $G_1$  and $G_2$ be two  groups. There
  exists  a bijection between  
  \begin{enumerate}
  \item 
    the set  of all subgroups  of the 
    direct  product  $G_1\times   G_2$,  and  
  \item 
    the  set   of  all  triples
    $(A/B,C/D,\psi )$ where 
    \begin{itemize}
    \item $A$ is a subgroup of $G_{1}$,
    \item $C$ is a subgroup of $G_{2}$,
    \item $B$ is a normal subgroup of $A$,
    \item $D$ is a normal subgroup of $C$, 
    \item $\psi: A/B\to C/D$ is a group isomorphism.
    \end{itemize}
  \end{enumerate}
  
\noindent Then, each subgroup of $ U \leq G_1\times G_2$ can be uniquely
  written as
  \begin{equation}\label{defU}
   U =  U_{\psi}= \{
      (a,c) \in A \times C 
      \mid
      (a + B) \psi =c + D
      \}.
  \end{equation}
\end{theorem}

Note that the isomorphism $\psi$ induces a homomorphism $\varphi: A \to C$ where $a\mapsto a\phi$  is such that $(a+B)\psi=a\varphi + D$ for any $a\in A$, and such that $B\varphi\leq D$. Such a homomorphism is not necessarily unique. 
\begin{corollary}
\label{lemma:psiforphi}
  Using notation of Theorem~\ref{gours}, given any homomorphism $\phi$ induced by $\psi$, we have 
 \begin{equation*}
    U_{\psi}
    =
    \{
      (a, a \varphi + d)
      \mid
      a \in A, d \in D
      \}.
  \end{equation*}
\end{corollary}
\begin{proof}
Let $(a,c) \in U_{\psi}.$ By definition of $\phi$, $(a+B)\psi = c+D = a\phi+D$, so $c \in a\phi+D$, and therefore there exists $d \in D$ such that $c = a \phi+d$.
Conversely, if $a \in A$ and $d \in D$, then $(a+B)\psi = a\phi +D = a\phi+d+D$.
\end{proof}

\subsection{Use of Goursat's Lemma}\label{rmkGoursat}
Let $U$ be a subspace of $V^4 = V^2 \times V^2$. 
From Theorem~\ref{gours} and Corollary~\ref{lemma:psiforphi} we have that there exist $A,B,C,D \leq V^2$ and $\psi: A/B \rightarrow C/D$ isomorphism inducing an homomorphism $\phi: A \to C$  such that
\[U = \{ (\bm{a}, \bm{a} \varphi + \bm{d}) \mid \bm{a} \in A, \bm{d} \in D\}.\]
 Without loss of generality, a basis of $A$ can be completed to a basis of $\F_2^{2n}$ and $\varphi$ can be arbitrarily defined from the basis of the complement $A^c$ of $A$ to a basis of $(\mathrm{Im}(\phi))^c$. Finally, $\varphi$ can be extended by linearity on the whole space $\F_2^{2n}$, providing us with a matrix representation of $\phi$ as 
\[
\begin{pmatrix} \phi_{11} & \phi_{12}\\ \phi_{21} & \phi_{22}\end{pmatrix}
\] 
such that for each $(a_1,a_2)\in A\leq\F_2^{2n}$
\[
\bm{a}\phi=(a_1,a_2)\phi = (a_1 \phi_{11} + a_2 \phi_{21}, a_1 \phi_{12}+a_2\phi_{22})\deq (\bm{a}\phi_{1},\bm{a}\phi_{2}),
\]
where, for $1 \leq i \leq 2$, 
\[
\phi_{i}=\begin{pmatrix} \phi_{1i} \\ \phi_{2i} \end{pmatrix}.
\] 

\noindent Applying again Goursat's Lemma on $A, D \leq V^2$, we obtain that 
\begin{enumerate}[(i)]
\item\label{des_A} there exist $A',B',C',D' \leq V$ and $\phi_A: A' \to C'$ an homomorphism such that 
\[A = \{ (a', a' \varphi_A + d') \mid a' \in A', d' \in D'\},\]
\item\label{des_D} there exist $A'',B'',C'',D'' \leq V$ and $\phi_D: A'' \to C''$ an homomorphism such that
\[D = \{ (a'', a'' \varphi_D + d'') \mid a'' \in A'', d'' \in D''\}.\]
\end{enumerate}
The previous construction and notation will be used in the remainder of the paper every time a subspace $U$ is considered as a candidate for the linear component of an invariant linear partition. More precisely:
\begin{definition}
A subgroup $  U \leq V^4$ is a \emph{linear block} for $f \in \Sym(V^4)$ if for each $\vec{v}\in V^4$ there exists $\vec{w} \in V^4$ such
that 
\begin{equation*}
(U +\vec{v})f = U+\vec{w},
\end{equation*} where we can always choose $\vec{w} = \vec{v}f.$ 
\end{definition}

When a linear block for ${f}$ is found, by Lemma~\ref{translatioBlocks} $\Span{f, T_{4n}}$ is imprimitive and have the cosets of the linear block as a block system.
By virtue of Lemma~\ref{translatioBlocks}, cosets of linear blocks are indeed the only kind of partitions that can be invariant for the groups under considerations, despite the generality of the definition of partition.  
Notice also that if $f \in \Sym(V^4)$ is such that $\vec 0f=\vec 0$ and $U < V^4$ is a linear block for $f$, then $U$ is an invariant subspace for $f$, i.e.\
for each $\vec u \in U$ there exists $\vec w \in U$ such that $\vec uf=\vec w$. The relation  $Uf= U$ can be also expressed, in the notation of this section, as 
\begin{equation}\label{lin_block}
\forall \bm{a} \in A \forall \bm{d} \in D \, \exists \bm{x} \in A \exists \bm{d} \in D: (\bm{a},\bm{a}\phi+\bm{d})f=(\bm{x},\bm{x}\phi+\bm{y}).
\end{equation}
We will use Eq.~\eqref{lin_block} extensively in the next results when considering functions with linear blocks, sometimes without explicit mention.

\subsection{The proof}
We are now ready to show the steps to prove Theorem~\ref{thm:main}. In the remainder of the paper we will make use of the notation introduced in Sec.~\ref{rmkGoursat} and we assume, without loss of generality and only for the  sake of simplicity, that $0\rho = 0$.
This is possible since each possible translation is considered in the group under investigation (cf. Definition~\ref{def:gr}).

The next result is the starting point for the proof of Theorem~\ref{thm:main}: we will show that assuming the existence of a linear block for $\rb$, i.e.\ exploiting an invariant subspace for $\rb$, leads to the discovery of a (possibly trivial) invariant subspace for $\rho$. Notice that our main claim follows straightforwardly from Lemma~\ref{lem_new} when such a subspace is non-trivial. In the reminder of the paper we will discuss the remaining cases. 

\begin{lemma}\label{lem_new}
Let $\rho \in \Sym(V)$ and let $U \leq V^4$ be a linear block for $\rb$. In the notation of Sec.~\ref{rmkGoursat} we have $  D''\rho=D''$.
\end{lemma}
\begin{proof}
\noindent Since $U$ is a linear block for $\rb$, taking $\bm{a}=\bm{0}$ in Eq.~\eqref{lin_block} and considering the description of $D$ as a subgroup of $\F_2^{2n}$ (cf.\ \ref{des_D} in Sec.~\ref{rmkGoursat}), for each $a'' \in A''$ and $d'' \in D''$, we have $(0,0,a'',a''\phi_D+d'')\in U$.
Moreover, assuming $a''=0$  and noticing that $U$ is a linear block for each element of $\Span{\rb} \leq \Sym(V^4)$, we have
 $
 (0,0,0,d'')\rb=(d''\rho,d''\rho,d''\rho,d''+d''\rho)\in U 
 $
 and
 $
 (0,0,0,d'')\rb^{-3}=(d''\rho,d''\rho,d''\rho,d'')\in U. 
 $
Therefore 
\begin{equation}\label{eq_new}
(d''\rho,d''\rho,d''\rho,d''+d''\rho)+(d''\rho,d''\rho,d''\rho,d'')=(0,0,0,d''\rho)\in U.
\end{equation} Hence, there exist 
  $\bm{x} \in A$ and $\bm{y} \in D$ such that $(0,0,0,d''\rho)=(\bm{x},\bm{x}\phi+\bm{y})$, and so $\bm{x}=\bm{0}$ and $(0,d''\rho)=\bm{y}\in D$. 
  From $(0,d''\rho)\in D$ we have that there exist $x''\in A''$ and $y''\in D''$ such that $x''=0$ and $d''\rho =y''\in D''$, which leads, since $\rho$ is a permutation, to
  $
  D''\rho=D'',
  $ 
  as claimed.
\end{proof}

We will use Lemma~\ref{lem_new} to prove that if $\Span{\rb,T_{4n}}$ is imprimitive, then an imprimitivity block for $\Span{\rho, T_n}$ can be found.
From Lemma~\ref{lem_new}, $D''$ is a natural first candidate for an imprimitivity block for $\Span{\rho, T_n}$.  The proof of Theorem~\ref{thm:main} is organized as follows: assuming that $U$ is an imprimitivity block for $\Span{\rb,T_{4n}}$, from Lemma~\ref{lem_new} we have that $D''$ is a block for $\rho$. When $D''$ is non-trivial and proper there is nothing left to prove. In the case $D'' = \F_2^n$
we derive a contradiction and in the case $D'' = \{0\}$ we prove that, instead, $C''$ is a block for  $\rho$. As before, the proof is completed when $C''$ is non-trivial and proper
and a contradiction is derived when $C'' = \F_2^n$. In the remaining case  $C'' = \{0\}$, $A'$ is proved to be a block for $\rho$, and the extremal possibilities for $A'$ are excluded by 
way of contradictions. In order to prove what anticipated, the following technical lemma is needed in some of the sub-cases.

\begin{lemma}\label{cond}
Let $\rho \in \Sym(V)$ and let $U \leq V^4$ be a linear block for $\rb$. In the notation of Sec.~\ref{rmkGoursat}, if $D=\{\bm{0}\}$ we have
\begin{enumerate}[(1)]
\item \label{item_Aphi=A} $A=A\phi$;
\item \label{item_a2inA'} if $(a_1,a_2)\in A$, then $a_1, a_2 \in A'$.
\end{enumerate}
\end{lemma}
\begin{proof}
Let us address each claim separately.
 Since $U$ is a linear block for $\rb$ such that $D=\{\bm{0}\}$, it means that
$U = \{ (\bm{a}, \bm{a} \varphi ) \mid \bm{a} \in A\}$
is a linear block also for $\rb^{-1}$ (cf.\ Definition~\ref{def:op} for the inverse of $\rb$) and, as in Eq.~\eqref{lin_block}, assuming $\bm{d}=(0,0)$ and $\bm{y}=(0,0)$, we have 
that for each $\bm{a}=(a_1,a_2) \in A$  there exists $\bm{x} \in A$  such that 
$
(\bm{a},\bm{a}\phi)\rb^{-1}=(\bm{x},\bm{x}\phi).
$
This means that
\begin{equation*}
\begin{split}
(a_1,a_2,\bm{a}\phi_1,\bm{a}\phi_2)\rb^{-1} &=(a_1 + (\bm{a}\phi_1+\bm{a}\phi_2)\rho, a_1+a_2,a_2+\bm{a}\phi_1,\bm{a}\phi_1+\bm{a}\phi_2)\\ 
& = (\bm{x},\bm{x}\phi).
\end{split}
\end{equation*}
Hence $\bm{x}\phi=(a_2+\bm{a}\phi_1,\bm{a}\phi_1+\bm{a}\phi_2) = (a_2,\bm{a}\phi_1)+(\bm{a}\phi_1,\bm{a}\phi_2) \in A\phi$.
Since $\bm{x}\phi$ and  $\bm{a}\phi$ belong to $A\phi$,  we have $(a_2,\bm{a}\phi_1)\in A\phi$.
Similarly, for each $\bm{a}=(a_1,a_2) \in A$, there exists $\bm{x} \in A$  such that 
\begin{equation*}
\begin{split}
(\bm{a},\bm{a}\phi)\rb^{-2} &=(a_1 + \xi +(a_2+\bm{a}\phi_2)\rho , a_2+ \xi,a_1+\bm{a}\phi_1,a_2+\bm{a}\phi_2)\\ 
& = (\bm{x},\bm{x}\phi),
\end{split}
\end{equation*}
where $\xi$ denotes $(\bm{a}\phi_1+\bm{a}\phi_2)\rho$.
Hence $(a_1+\bm{a}\phi_1,a_2+\bm{a}\phi_2)=(a_1,a_2)+(\bm{a}\phi_1,\bm{a}\phi_2)\in A\phi$ and so $(a_1,a_2)\in A\phi$, which proves $A\leq A\phi$ and, from $\Size{A}\geq\Size{A\phi}$, we obtain claim~\ref{item_Aphi=A}.
 Moreover, since we have already proved that $(a_2,\bm{a}\phi_1)\in A\phi=A$,  by the description of $A$ as subgroup of $\F_2^{2n}$ (cf.\ \ref{des_A} in Sec.~\ref{rmkGoursat}), there exist $x'\in A'$ and $y'\in D'$ such that $(a_2,\bm{a}\phi_1)=(x',x'\phi_A+y')$, and so $a_2=x'\in A'$. Similarly, for each $(a_1,a_2)\in A$ we have $a_1\in A'$, i.e.\ claim~\ref{item_a2inA'} is obtained. \qedhere
\end{proof}

We now use the previous lemma to show our main result, in which we prove that, in general, the AES-like key-schedule construction generates a primitive permutation group,
provided that the key-schedule operator $\rb$ is induced by a permutation $\rho$ such that $\Span{\rho, T_n}$ is primitive. As already anticipated, the proof is organized in several steps. We will proceed as described in the paragraph after Lemma~\ref{cond}.

\begin{proof}[Proof of Theorem~\ref{thm:main}]
Let us assume that $\Span{\rb,T_{4n}}$ is imprimitive, i.e.\ that there exists
a block system $\mathcal U$ for $\Span{\rb, T_{4n}}$. Then, from Lemma~\ref{translatioBlocks},  the block system 
is of the type
 \[
 \mathcal U = \{U + \vec{v} \mid \vec{v}\in V^4\}
 \]
for a non-trivial and proper subspace $U$ of $V^4$. From Lemma~\ref{lem_new}  we have $  D''\rho=D''$ and
 the previous equation means that the subgroup $D'' \leq V$, when non-trivial and proper, is an imprimitivity block for $\Span{\rho,T_n}$. 
 If that is the case, there is nothing left to prove. 
 Let us conclude the proof addressing the extremal cases $D'' = \F_2^n$ and $D'' = \{0\}$ separately. 
 \medskip 
\begin{description}
	\item[\mybox{$\mathbf{D'' = \F_2^n}$}] \input{files/Dsec_tutto.tex}
	\smallbreak
	\item[\mybox{$\mathbf{D'' = \{0\}}$}] \input{files/Dsec_zero.tex}
		\medskip 
		\begin{description}
			\item[\mybox{$\mathbf{C'' = \F_2^n}$}] \input{files/Csec_tutto.tex}
			\smallbreak
			\item[\mybox{$\mathbf{C'' = \{0\}}$}] \input{files/Csec_zero.tex}
				 \medskip 
				\begin{description}
					\item[\mybox{$\mathbf{A' = \F_2^n}$}] \input{files/Apri_tutto.tex}
					 \smallbreak 
					\item[\mybox{$\mathbf{A' = \{0\}}$}] \input{files/Apri_zero.tex}
				\end{description}
		\end{description}
\end{description}
 \end{proof} 
\begin{remark}
Notice that in Theorem~\ref{thm:main} we have obtained our claim by reaching the contradiction that $D''$ (or $C''$ or $A'')$ is an invariant subspace for $\rho$. We should actually prove that $D''$ generates an invariant partition. However, computations are nearly identical and identically tedious and therefore are not included in this presentation. The intrigued reader may find the same results rewriting the proof of Theorem~\ref{thm:main} obtaining that $(D''+v) \mapsto D''+w$ for some $w \in \mathbb F_2^n$. 
\end{remark}

\section{Conclusions}\label{sec:concl}
In this work we have considered the group $\Gamma_{\mathrm{AES}} = \Span{\overline{\rho_{\mathrm{AES}}}, T_{128}}$ generated by the AES-128 key-schedule transformations and we have proved that no partition of $V^4 = \F_2^{128}$ can be invariant under its action. However, the slow global diffusion of the operator  does not suffice to make the key-schedule  transformation free from invariant linear partitions when the composition of more rounds is considered. 
In particular, since $\lambda^2$ and $\lambda^4$ admit proper and non-trivial invariant subspaces which are a direct sum of bricks of $V$, we can conclude that group generated by  $i$ consecutive key-schedule transformations $\Span{\overline{\rho_{\mathrm{AES}}}^{\,i}, T_{128}}$
is 
\begin{itemize}
\item[-] primitive if $i = 1$ (this work) and
\item[-] imprimitive if $i \in \{0,2 \textrm{ mod } 4\}$ (see e.g.~\cite[Proposition 5.1]{carantiprimitive} or~\cite{caldenote} and~\cite{leurentnew}).
\end{itemize}
It comes then with no surprise  that $\overline{\rho_{\mathrm{AES}}}^4$ admits invariant subspaces, like those found by
Leurent and Pernot~\cite{leurentnew}, using an
algorithm of Leander et al.~\cite{leander2015generic}. One example is $U < V^4$,
where
\[
U \deq \{(a,b,c,d,0,b,0,d,a,0,0,d,0,0,0,d) \mid a,b,c,d \in \F_2^8 \}.
\]

Although the results of this work are not straightforwardly generalized using the same methods to the case $i =3$, we find it easy to believe that
also $\Span{\overline{\rho_{\mathrm{AES}}}^{\,3}, T_{128}}$ act primitively on $V^4$. Moreover, there
is no reason to believe that the same result is not valid for the 192-bit and 256-bit versions of AES key-schedule. However, the increasing complexity 
of the strategy used here does not seem to be suitable for addressing the problem, which might require a different methodology. 

\section*{Acknowledgment}
The authors are gratefully thankful to the referees for their valuable and constructive corrections and suggestions that have
improved the quality of the manuscript.
\bibliographystyle{alpha}
\bibliography{sym2n_ref}

\end{document}

%% file: files/preamble.tex
\usepackage{amssymb,amsmath}
\usepackage[english]{babel}
\usepackage{amsfonts}
\usepackage{color}
\usepackage{amsthm}
\usepackage[colorlinks=true,linkcolor=NavyBlue, citecolor=NavyBlue,urlcolor=NavyBlue]{hyperref}
\usepackage{graphicx}
\usepackage{mathtools}

\usepackage[usenames,dvipsnames]{pstricks}
 \usepackage{epsfig}
\usepackage{pst-grad} 
\usepackage{pst-plot} 

\usepackage{collectbox}
\makeatletter
\newcommand{\mybox}{%
    \collectbox{%
        \setlength{\fboxsep}{2pt}%
        \fbox{\BOXCONTENT}%
    }%
}

\makeatother

\usepackage{relsize}

\usepackage{bm}
\usepackage{amsmath}
\usepackage{listings}
\usepackage{hyperref}
\usepackage[shortlabels]{enumitem}

\usepackage{xy}
\usepackage[draft]{fixme}
\DeclareMathOperator{\AGL}{AGL}
\DeclareMathOperator{\GL}{GL}
\DeclareMathOperator{\Sym}{Sym}

\DeclareMathOperator{\F}{\mathbb{F}}

\newcommand\deq{\mathrel{\stackrel{\makebox[0pt]{\mbox{\normalfont\tiny def}}}{=}}}

\newcommand{\Size}[1]{\left\lvert #1 \right\rvert}
\newcommand{\Span}[1]{\left\langle#1\right\rangle}

\let\phi\varphi
\theoremstyle{plain}
\newtheorem{theorem}{Theorem}[section]

\newtheorem{lemma}[theorem]{Lemma}

\newtheorem{corollary}[theorem]{Corollary}

\theoremstyle{remark}
\newtheorem{remark}{Remark}

\theoremstyle{definition}
\newtheorem{definition}[theorem]{Definition}

\newtheorem*{notation*}{Notation}

\usepackage{listings}
\lstdefinelanguage{GAP}{%
 morekeywords={%
 Assert,Info,IsBound,QUIT,%
 TryNextMethod,Unbind,and,break,%
 continue,do,elif,%
 else,end,false,fi,for,%
 function,if,in,local,%
 mod,not,od,or,%
 quit,rec,repeat,return,%
 then,true,until,while%
 },%
 sensitive,%
 morecomment=[l]\#,%
 morestring=[b]",%
 morestring=[b]',%
}[keywords,comments,strings]

\usepackage[T1]{fontenc}
\usepackage[variablett]{lmodern}
\usepackage{xcolor}
\DeclareMathAlphabet{\mbb}{U}{bbold}{m}{n}
\let \varphi \phi

\newcommand{\rb}{\overline{\rho}}

\newcommand{\M}{M}

%% file: files/Dsec_tutto.tex
 Since $D'' \leq A''\phi_D$, then $A''\phi_D= \F_2^n$ and, from $\Size{A''}\geq\Size{A''\phi_D}$, we also have $A''= \F_2^n$. Therefore $B''=C'' = \F_2^n$, since by hypothesis $A''/B'' \cong C''/D''$. This proves that $D=\F_2^{2n}$. From \ref{item_Aphi=A} of  Lemma~\ref{cond}, we obtain $A=\F_2^{2n}$, and so $B=C= \F_2^{2n}$, since by hypothesis $A/B \cong C/D$. This proves that $U$ is not proper, a contradiction.

%% file: files/Dsec_zero.tex
Let us prove first that, in this case, also $B''=\{0\}$. Indeed, since $B''\phi_D \leq D''$ and $D''=\{0\}$, we have $B''\phi_D=\{0\}$. If we set $\bm{a}=\bm{0}$ and $a''=b''\in B''$ in Eq.~\eqref{lin_block}, then we have $(0,0,b'',0)\in U$, since $b''\phi_D=0$. Moreover, we have that  
 \[
 (0,0,b'',0)\rb=(0,0,b'',b'')\in U,
 \] 
and so $(0,0,b'',b'')+(0,0,b'',0)=(0,0,0,b'')\in U$, which implies
\[(0,0)\phi+(0,b'')=(0,b'')\in D,\] and so there exists $x''\in A''$ such that $(0,b'')=(x'',x''\phi_D)$, from which $0=0\phi_D=b''$, i.e.\ $B''=\{0\}$.
This also proves that $\phi_D:A'' \rightarrow C''$ is an isomorphism.\\
Now, setting $\bm{a}=\bm{0}$, we have
\[
(0,0,a'',a''\phi_D)\rb=(a''\phi_D\rho,a''\phi_D\rho,a''+a''\phi_D\rho,a''+a''\phi_D+a''\phi_D\rho)\in U
\]
and
\[
(0,0,a'',a''\phi_D)\rb^{-3}=(a''\phi_D\rho,a''\phi_D\rho,a''+(a''+a''\phi_D)\rho,a''+a''\phi_D)\in U.
\]
Therefore there exist $\bm{x}\in A$ and $\bm{y}\in D$ such that
\begin{equation*}
\begin{split}
(0,0,a'',a''\phi_D)\rb+(0,0,a'',a''\phi_D)\rb^{-3}&= (0,0,a''\phi_D\rho+(a''+a''\phi_D)\rho,a''\phi_D\rho)\\
&= (\bm{x},\bm{x}\phi+\bm{y}),
\end{split}
\end{equation*}
and so $(a''\phi_D\rho+(a''+a''\phi_D)\rho,a''\phi_D\rho)\in D$. Hence there exists $x''\in A''$ such that $x''\phi_D+a''\phi_D\rho\in A''\phi_D$ and so $a''\phi_D\rho\in A''\phi_D$. This proves that $A''\phi_D\rho=A''\phi_D$, since $\rho$ is a permutation. Moreover, since $\phi_D$ is an isomorphism, we have $C''\rho=C''$. If $C''$ is a non-trivial and proper subgroup of $V$, then we have determined another imprimitivity block for $\Span{\rho,T_n}$, so the claim is proved. 
 Let us address the extremal cases $C'' = \F_2^n$ and $C'' = \{0\}$ separately. 
 

%% file: files/Csec_tutto.tex
First notice that $A''=\F_2^{n}$ since, as already proved, $\phi_D$ is an isomorphism. From \ref{item_A''<A'} and \ref{item_A''<D'} of Lemma~\ref{cond}, we have that $A'=D'=\F_2^{n}$, and so  $B'=C'= \F_2^{n}$, since by hypothesis $A'/B' \cong C'/D'$. This proves that $A\phi=A=\F_2^{2n}$. Since $\Size{C}\geq\Size{A\phi}$, then $C=\F_2^{2n}$ and $\phi$ is an automorphism of $A=\F_2^{2n}$. From $\mathrm{Ker}(\phi)=\{\bm{0}\}$, it follows that $B=\{\bm{0}\}$, since $B\leq \mathrm{Ker}(\phi)$. Finally, since $A/B \cong C/D$, we also have  $D=\{\bm{0}\}$, which contradicts the fact that $\phi_D$ is an isomorphism and  $C'' = \F_2^n$.

%% file: files/Csec_zero.tex
Since $\phi_D$ is an isomorphism, we have $C''=D''=B''=A''=\{0\}$, and so $D=\{\bm{0}\}$. Let us now prove that $B=\{\bm{0}\}$. Since $B\phi \leq D$ and $D=\{\bm{0}\}$, then $B\phi=\{\bm{0}\}$. If $(b_1,b_2)\in B$, then $(b_1,b_2)\phi=(0,0)$, and so $(b_1,b_2,0,0)\in U$. Therefore we have
 \[
 (b_1,b_2,0,0)\rb=(b_1,b_1+b_2,b_1+b_2,b_1+b_2)\in U
 \]
 and
 \[
  (b_1,b_2,0,0)\rb^{-1}=(b_1,b_1+b_2,b_2,0)\in U,
 \]
 and so $(b_1,b_1+b_2,b_1+b_2,b_1+b_2)+(b_1,b_1+b_2,b_2,0)=(0,0,b_1,b_1+b_2)\in U$. Therefore, there exists $\bm{x}\in A$ such that $(0,0,b_1,b_1+b_2)=(\bm{x},\bm{x}\phi)$, so we have $(b_1,b_1+b_2)=(0,0)$, which implies $(b_1,b_2)=\bm{0}$. This proves that $B=\{\bm{0}\}$ and that $\phi:A\rightarrow C$ is an isomorphism. From \ref{item_Aphi=A} of Lemma~\ref{cond}, we have that $\phi$ is an automorphism of $A$. Moreover, for each $\bm{a}=(a_1,a_2)\in A$, we have $(a_1,a_2)\phi=(\bm{a}\phi_1,\bm{a}\phi_2)\in A\phi=A$, and by \ref{item_a2inA'} of Lemma~\ref{cond} we obtain $\bm{a}\phi_1,\bm{a}\phi_2\in A'$, and so $\bm{a}\phi_1+\bm{a}\phi_2\in A'$. Consequently, 
 \[
 \mathrm{Im}(\phi_1+\phi_2)=\{\bm{a}\phi_1+\bm{a}\phi_2 \mid \bm{a}\in A\}\leq A'.
 \]
 Notice that $\phi_1+\phi_2$ is surjective, since $\phi=(\phi_1,\phi_2)$ is an invertible matrix, and so $\mathrm{Im}(\phi_1+\phi_2)=A'$.\\
 Now, for each $\bm{a}=(a_1,a_2)\in A$, there exists $\bm{x}\in A$ such that
\begin{equation*}
\begin{split}
(\bm{a},\bm{a}\phi)\rb^{-2} &= (a_1 + \xi  +(a_2+\bm{a}\phi_2)\rho , a_2+ \xi,a_1+\bm{a}\phi_1,a_2+\bm{a}\phi_2)\\ 
& = (\bm{x},\bm{x}\phi),
\end{split}
\end{equation*}
thus we obtain 
\[
(a_1 + \xi +(a_2+\bm{a}\phi_2)\rho , a_2+ \xi)\phi=(a_1,a_2)+(a_1,a_2)\phi \in A\phi=A,
\]
 where $\xi$ denotes here the element $(\bm{a}\phi_1+\bm{a}\phi_2)\rho$.\\
 Therefore
\begin{equation}\label{eq:phi2-1}
((\bm{a}\phi_1+\bm{a}\phi_2)\rho +(a_2+\bm{a}\phi_2)\rho , (\bm{a}\phi_1+\bm{a}\phi_2)\rho)=(a_1,a_2)\phi^{-1}\in A\phi=A.
\end{equation}
Therefore, from   \ref{item_a2inA'} of Lemma~\ref{cond}, for each $\bm{a}\in A$ we have $(\bm{a}\phi_1+\bm{a}\phi_2)\rho)\in A'$, and so we obtain $A'\rho=A'$, since
\[
 \mathrm{Im}(\phi_1+\phi_2)=\{\bm{a}\phi_1+\bm{a}\phi_2 \mid \bm{a}\in A\}= A'
 \]
and $\rho$ is a permutation. As before, the proof is completed when $A'$ is a non-trivial and proper subgroup of $V$, since it represents an imprimitivity block for $\Span{\rho,T_n}$. Otherwise, the following two cases remain to be discussed.

%% file: files/Apri_tutto.tex
Let us denote by
 \[
 \theta\deq \phi^{-1} = \begin{pmatrix} \theta_{11} &\theta_{12} \\ \theta_{21} & \theta_{22}\end{pmatrix}
 \] 
 and let us denote by
 \[
 \theta_1\deq\begin{pmatrix}\theta_{11}  \\ \theta_{21}\end{pmatrix}
 \text{ and } 
 \theta_2\deq\begin{pmatrix}\theta_{12} \\  \theta_{22}\end{pmatrix}.
 \]
Notice that, from Eq.~\eqref{eq:phi2-1}, 
we have
\[
(\bm{a}\phi_1+\bm{a}\phi_2)\rho=\bm{a}\theta_2,
\]
 which implies that $\rho$ is linear on $\{\bm{a}\phi_1+\bm{a}\phi_2 \mid \bm{a}\in A\}= A'=\F_2^{n}$, a contradiction.

%% file: files/Apri_zero.tex
First notice that $A'\phi_A=\{0\}$ and also $D'=\{0\}$, since $D'\leq A'\phi_A$. Hence $B'=C'= \{0\}$, since by hypothesis $A'/B' \cong C'/D'$, and so $A=\{\bm{0}\}$. 	Finally, since $\phi$ is an automorphism of $A$, we have $C=\{\bm{0}\}$, and so $U=\{\vec{0}\}$, a contradiction.

%% file: AESks_rev2.bbl
\newcommand{\etalchar}[1]{$^{#1}$}
\begin{thebibliography}{MDRMH10}

\bibitem[AC21]{aragona2020primitivity}
Riccardo Aragona and Roberto Civino.
\newblock On invariant subspaces in the {L}ai--{M}assey scheme and a
  primitivity reduction.
\newblock {\em Mediterranean Journal of Mathematics}, 18(4):1--14, 2021.

\bibitem[ACC{\etalchar{+}}19]{aragonawave}
Riccardo Aragona, Marco Calderini, Roberto Civino, Massimiliano Sala, and
  Ilaria Zappatore.
\newblock Wave-shaped round functions and primitive groups.
\newblock {\em Adv. Math. Commun.}, 13(1):67--88, 2019.

\bibitem[ACC20]{arafeistel}
Riccardo Aragona, Marco Calderini, and Roberto Civino.
\newblock Some group-theoretical results on {F}eistel networks in a long-key
  scenario.
\newblock {\em Adv. Math. Commun.}, 14(4):727--743, 2020.

\bibitem[ACS17]{aragona2017group}
Riccardo Aragona, Andrea Caranti, and Massimiliano Sala.
\newblock The group generated by the round functions of a {GOST}-like cipher.
\newblock {\em Ann. Mat. Pura Appl. (4)}, 196(1):1--17, 2017.

\bibitem[ACTT18]{aragona2018primitivity}
Riccardo Aragona, Marco Calderini, Antonio Tortora, and Maria Tota.
\newblock Primitivity of {PRESENT} and other lightweight ciphers.
\newblock {\em J. Algebra Appl.}, 17(6):1850115, 2018.

\bibitem[BCLR17]{beierle2017proving}
Christof Beierle, Anne Canteaut, Gregor Leander, and Yann Rotella.
\newblock Proving resistance against invariant attacks: How to choose the round
  constants.
\newblock In {\em Advances in Cryptology---{CRYPTO} 2017. {P}art {II}}, volume
  10402 of {\em Lecture Notes in Comput. Sci.}, pages 647--678. Springer, Cham,
  2017.

\bibitem[BK09]{biryukov2009related}
Alex Biryukov and Dmitry Khovratovich.
\newblock Related-key cryptanalysis of the full {AES}-192 and {AES}-256.
\newblock In {\em Advances in Cryptology---{ASIACRYPT} 2009}, volume 5912 of
  {\em Lecture Notes in Comput. Sci.}, pages 1--18. Springer, Berlin, 2009.

\bibitem[BLNPS18]{boura2018making}
Christina Boura, Virginie Lallemand, Mar\'{\i}a Naya-Plasencia, and Valentin
  Suder.
\newblock Making the impossible possible.
\newblock {\em J. Cryptology}, 31(1):101--133, 2018.

\bibitem[BR19]{bardeh2019exchange}
Navid~Ghaedi Bardeh and Sondre R{\o}njom.
\newblock The exchange attack: How to distinguish six rounds of {AES} with
  $2^{88.2}$ chosen plaintexts.
\newblock In {\em Advances in Cryptology---{ASIACRYPT} 2019. {P}art {III}},
  volume 11923 of {\em Lecture Notes in Comput. Sci.}, pages 247--370.
  Springer, Cham, 2019.

\bibitem[Cal18]{caldenote}
Marco Calderini.
\newblock A note on some algebraic trapdoors for block ciphers.
\newblock {\em Adv. Math. Commun.}, 12(3):515--524, 2018.

\bibitem[Cal20]{calderinikey}
Marco Calderini.
\newblock Primitivity of the group of a cipher involving the action of the
  key-schedule.
\newblock {\em J. Algebra Appl.}, Online {R}eady:21500845, 2020.

\bibitem[Cam99]{cameronpermutation}
Peter~J. Cameron.
\newblock {\em Permutation groups}, volume~45 of {\em London Mathematical
  Society Student Texts}.
\newblock Cambridge University Press, Cambridge, 1999.

\bibitem[CCZ98]{carlet1998codes}
Claude Carlet, Pascale Charpin, and Victor Zinoviev.
\newblock Codes, bent functions and permutations suitable for {DES}-like
  cryptosystems.
\newblock {\em Des. Codes Cryptogr.}, 15(2):125--156, 1998.

\bibitem[CDVS09]{carantiprimitive}
Andrea Caranti, Francesca Dalla~Volta, and Massimiliano Sala.
\newblock On some block ciphers and imprimitive groups.
\newblock {\em Appl. Algebra Engrg. Comm. Comput.}, 20(5-6):339--350, 2009.

\bibitem[DKRS20]{dunkelman2020retracing}
Orr Dunkelman, Nathan Keller, Eyal Ronen, and Adi Shamir.
\newblock The retracing boomerang attack.
\newblock In {\em Advances in Cryptology---{EUROCRYPT} 2020. {P}art {I}},
  volume 12105 of {\em Lecture Notes in Comput. Sci.}, pages 280--309.
  Springer, Cham, 2020.

\bibitem[DR02]{aes}
Joan Daemen and Vincent Rijmen.
\newblock {\em The design of {R}ijndael}.
\newblock Information Security and Cryptography. Springer-Verlag, Berlin, 2002.

\bibitem[Gou89]{goursat}
Edouard Goursat.
\newblock Sur les substitutions orthogonales et les divisions r\'{e}guli\`eres
  de l'espace.
\newblock {\em Ann. Sci. \'{E}cole Norm. Sup. (3)}, 6:9--102, 1889.

\bibitem[GRR17a]{grassi2017new}
Lorenzo Grassi, Christian Rechberger, and Sondre R{\o}njom.
\newblock A new structural-differential property of 5-round {AES}.
\newblock In {\em Advances in Cryptology---{EUROCRYPT} 2017. {P}art {II}},
  volume 10211 of {\em Lecture Notes in Comput. Sci.}, pages 289--317.
  Springer, Cham, 2017.

\bibitem[GRR17b]{grassi2016subspace}
Lorenzo Grassi, Christian Rechberger, and Sondre Rønjom.
\newblock Subspace trail cryptanalysis and its applications to {AES}.
\newblock {\em IACR Transactions on Symmetric Cryptology}, 2016(2):192--225,
  2017.

\bibitem[LAAZ11]{leander2011cryptanalysis}
Gregor Leander, Mohamed~Ahmed Abdelraheem, Hoda AlKhzaimi, and Erik Zenner.
\newblock A cryptanalysis of {PRINT}cipher: the invariant subspace attack.
\newblock In {\em Advances in Cryptology---{CRYPTO} 2011}, volume 6841 of {\em
  Lecture Notes in Comput. Sci.}, pages 206--221. Springer, Heidelberg, 2011.

\bibitem[LMR15]{leander2015generic}
Gregor Leander, Brice Minaud, and Sondre R{\o}njom.
\newblock A generic approach to invariant subspace attacks: Cryptanalysis of
  {R}obin, i{SCREAM} and {Z}orro.
\newblock In {\em Advances in Cryptology---{EUROCRYPT} 2015. {P}art {I}},
  volume 9056 of {\em Lecture Notes in Comput. Sci.}, pages 254--283. Springer,
  Heidelberg, 2015.

\bibitem[LP21]{leurentnew}
Ga{\"e}tan Leurent and Clara Pernot.
\newblock New representations of the {AES} {K}ey {S}chedule.
\newblock In {\em Advances in Cryptology---{EUROCRYPT} 2021. {P}art {I}},
  volume 12696 of {\em Lecture Notes in Comput. Sci.}, pages 54--84. Springer,
  Cham, 2021.

\bibitem[MDRMH10]{mala2010improved}
Hamid Mala, Mohammad Dakhilalian, Vincent Rijmen, and Mahmoud Modarres-Hashemi.
\newblock Improved impossible differential cryptanalysis of 7-round {AES}-128.
\newblock In {\em Progress in Cryptology---{INDOCRYPT} 2010}, volume 5498 of
  {\em Lecture Notes in Comput. Sci.}, pages 282--291. Springer, Berlin, 2010.

\bibitem[Nyb93]{nyberg1993differentially}
Kaisa Nyberg.
\newblock Differentially uniform mappings for cryptography.
\newblock In {\em Advances in Cryptology---{EUROCRYPT} 1993}, volume 765 of
  {\em Lecture Notes in Comput. Sci.}, pages 55--64. Springer, Berlin, 1993.

\bibitem[Pat99]{paterson1999imprimitive}
Kenneth~G. Paterson.
\newblock {I}mprimitive {P}ermutation {G}roups and {T}rapdoors in {I}terated
  {B}lock {C}iphers.
\newblock In {\em Fast Software Encryption}, volume 1636 of {\em Lecture Notes
  in Comput. Sci.}, pages 201--214. Springer, Berlin, 1999.

\bibitem[RBH17]{ronjom2017yoyo}
Sondre R{\o}njom, Navid~Ghaedi Bardeh, and Tor Helleseth.
\newblock Yoyo tricks with {AES}.
\newblock In {\em Advances in Cryptology---{ASIACRYPT} 2017. {P}art {I}},
  volume 10624 of {\em Lecture Notes in Comput. Sci.}, pages 217--243.
  Springer, Cham, 2017.

\bibitem[SW08]{sparwenrij}
R\"{u}diger Sparr and Ralph Wernsdorf.
\newblock Group theoretic properties of {R}ijndael-like ciphers.
\newblock {\em Discrete Appl. Math.}, 156(16):3139--3149, 2008.

\bibitem[SW15]{wenkas}
R\"{u}diger Sparr and Ralph Wernsdorf.
\newblock The round functions of {KASUMI} generate the alternating group.
\newblock {\em J. Math. Cryptol.}, 9(1):23--32, 2015.

\bibitem[Wer93]{wendes}
Ralph Wernsdorf.
\newblock The one-round functions of the {DES} generate the alternating group.
\newblock In {\em Advances in Cryptology---{EUROCRYPT} 1992}, volume 658 of
  {\em Lecture Notes in Comput. Sci.}, pages 99--112. Springer, Berlin, 1993.

\end{thebibliography}
